\documentclass[a4paper,12pt]{article}

\usepackage[left=2cm,right=2cm, top=2cm,bottom=3cm,bindingoffset=0cm]{geometry}

\usepackage{verbatim}
\usepackage{amsmath}
\usepackage{amsthm}
\usepackage{amssymb}
\usepackage{delarray}
\usepackage{cite}
\usepackage{hyperref}
\usepackage{mathrsfs}
\usepackage{tikz}
\usetikzlibrary{patterns}
\usepackage{caption}
\DeclareCaptionLabelSeparator{dot}{. }
\newcommand{\al}{\alpha}
\newcommand{\be}{\beta}
\newcommand{\ga}{\gamma}
\newcommand{\de}{\delta}
\newcommand{\la}{\lambda}

\newcommand{\eps}{\varepsilon}
\newcommand{\vv}{\varphi}
\newcommand{\iy}{\infty}

\theoremstyle{plain}

\numberwithin{equation}{section}

\newtheorem{thm}{Theorem}[section]
\newtheorem{lem}[thm]{Lemma}

\theoremstyle{definition}

\newtheorem{alg}[thm]{Algorithm}
\newtheorem{ip}[thm]{Inverse Problem}

\theoremstyle{remark}

\DeclareMathOperator*{\Res}{Res}

 \allowdisplaybreaks

\begin{document}

\begin{center}
{\large\bf Solving an inverse problem for the Sturm-Liouville operator with a singular potential by Yurko's method}
\\[0.2cm]
{\bf Natalia P. Bondarenko} \\[0.2cm]
\end{center}

\vspace{0.5cm}

{\bf Abstract.} An inverse spectral problem for the Sturm-Liouville operator with a singular potential from the class $W_2^{-1}$ is solved by the method of spectral mappings. We prove the uniqueness theorem, develop a constructive algorithm for solution and obtain necessary and sufficient conditions of solvability for the inverse problem in the self-adjoint and the non-self-adjoint cases.

\medskip

{\bf Keywords:} inverse spectral problems; Sturm-Liouville operator; singular potential; method of spectral mappings.

\medskip

{\bf AMS Mathematics Subject Classification (2010):} 34A55 34B09 34B24 34L40 

\vspace{1cm}

\section{Introduction}

This paper concerns the theory of inverse spectral problems. Such problems consist in recovering differential operators from their spectral characteristics. Inverse spectral problems have applications in quantum mechanics, geophysics, chemistry, electronics and other branches of science and engineering (see, e.g., \cite{FY01} and references therein). The basic results of the inverse problem theory were obtained for operators induced by the Sturm-Liouville expression $\ell y := -y'' + q(x) y$ with the \textit{regular} (i.e., square integrable) potential $q$ (see \cite{FY01, Mar77, Lev84, PT87}). In recent years, spectral analysis of differential operators with {\it singular} coefficients from spaces of distributions has attracted much attention of mathematicians (see \cite{SS99, Sav01, SS01, HM-trans, SS06, HM-asympt, Mirz14, MS16, KMS18, KM19, HM-sd, HM-2sp, HM-half, SS05, HM06, FIY08, SS08, DM09, MT09, SS10, Hryn11, HP12, Pron13, Sav16, Bond18, Bond19, Gul19, Vas19}). Properties of spectral characteristics and solutions of differential equations with singular coefficients were studied in \cite{SS99, Sav01, SS01, HM-trans, SS06, HM-asympt, Mirz14, MS16, KMS18, KM19}. Some aspects of inverse problem theory for differential operators with singular coefficients have been investigated in \cite{HM-sd, HM-2sp, HM-half, SS05, HM06, FIY08, SS08, DM09, MT09, SS10, Hryn11, HP12, Pron13, Sav16, Bond18, Bond19, Gul19, Vas19}. Nevertheless, a number of open questions still remain in this field.

In this paper, we consider the differential expression $\ell y = -y'' + q(x) y$ with the singular complex-valued potential $q$ from the class $W_2^{-1}(0, \pi)$. This means that $q = \sigma'$, where $\sigma$ is a function from $L_2(0, \pi)$, and the derivative is understood in the sense of distributions. Then the differential expression $\ell y$ can be represented in the following form:
$$
\ell y = -(y^{[1]})' - \sigma(x) y^{[1]} - \sigma^2(x) y,
$$
where $y^{[1]}(x) := y'(x) - \sigma(x) y(x)$ is the so-called {\it quasi-derivative}.

Consider the boundary value problem
\begin{gather} \label{eqv}
    \ell y = \la y, \quad x \in (0, \pi), \\ \nonumber
    y^{[1]}(0) - h y(0) = 0, \quad y^{[1]}(\pi) + H y(\pi) = 0,
\end{gather}
where $h$ and $H$ are complex constants, functions $y$, $y^{[1]}$ are absolutely continuous on $[0, \pi]$ and $(y^{[1]})' \in L_2(0, \pi)$. Without loss of generality we assume that $h = 0$. One can easily achieve this condition by the shift $\sigma(x) := \sigma(x) + h$. Thus, we denote by $L = L(\sigma, H)$ the boundary value problem for equation~\eqref{eqv} with the boundary conditions
\begin{equation} \label{bc}
   y^{[1]}(0) = 0, \quad y^{[1]}(\pi) + H y(\pi) = 0. 
\end{equation}

Let $\vv(x, \la)$ be the solution of equation~\eqref{eqv} satisfying the initial conditions $\vv(0, \la) = 1$, $\vv^{[1]}(0, \la) = 0$. The functions $\vv(x, \la)$, $\vv^{[1]}(x, \la)$ are entire by $\la$ for each fixed $x \in [0, \pi]$. The spectrum of $L$ is a countable set of complex eigenvalues $\{ \la_n \}_{n = 0}^{\iy}$ ($|\la_n| \le |\la_{n + 1}|$, $n \ge 0$). The eigenvalues coincide with the zeros of the characteristic function $\Delta(\la) := \vv^{[1]}(\pi, \la) + H \vv(\pi, \la)$. For simplicity, we assume that all the eigenvalues are simple, i.e., $\la_n \ne \la_k$ for $n \ne k$. The case of multiple eigenvalues can be treated similarly to \cite{But07}. 

Define the {\it weight numbers} as follows:
$$
\al_n := \left( \int_0^{\pi} \vv^2(x, \la_n) \, dx\right)^{-1}, \quad n \ge 0.
$$
The numbers $\{ \la_n, \al_n \}_{n = 0}^{\iy}$ are called the {\it spectral data} of $L$. The paper is devoted to the following problem.

\begin{ip} \label{ip:1}
Given the spectral data $\{ \la_n, \al_n \}_{n \ge 0}$, find $\sigma$ and $H$.
\end{ip}

The goal of this paper is to solve Inverse Problem~\ref{ip:1} by the {\it method of spectral mappings} developed by Yurko (see \cite{FY01, Yur02}). We prove the uniqueness theorem, obtain a constructive algorithm of solution together with necessary and sufficient conditions of solvability for Inverse Problem~\ref{ip:1} in the self-adjoint and the non-self-adjoint cases.

The method of spectral mappings consists in reduction of a nonlinear inverse problem to a linear equation in a Banach space (the so-called \textit{main equation}). This reduction is based on the Cauchy-Poincar\'e contour integration in the complex plane of the spectral parameter. These ideas first appeared in the papers of Levinson \cite{Lev49} and Leibenson \cite{Leib66, Leib71}. Later on, the method of spectral mappings allowed Yurko and his followers to solve inverse problems for higher-order differential operators, for differential systems, for differential operators on geometrical graphs, and for other operator classes appearing in applications (see \cite{FY01, Yur02, Yur05, Yur16} and references therein). We note that, for different classes of differential operators, the construction of the main equation and of an appropriate Banach space differs and may require significant modifications comparing with the classical Sturm-Liouville operator.

For the Sturm-Liouville inverse problems with singular potentials, the majority of the known results were obtained by development of the Gelfand-Levitan method (see \cite{HM-sd, HM-2sp, HM-half, Hryn11}) and of the Trubowitz approach (see \cite{SS05, SS08, SS10, Sav16}). The results of all those papers related to inverse problems concern only the self-adjoint case, in which $\sigma$ is real-valued.
Some ideas of Yurko's method were applied to some special issues of inverse problem theory for differential operators with singular coefficients on graphs in \cite{FIY08, Bond19, Vas19}. Nevertheless, it was unknown how to obtain necessary and sufficient conditions of inverse problem solvability, being a crucial issue of inverse problem theory, by using the method of spectral mappings. This paper aims to cover this gap. 

To present our ideas, we choose the simplest problem~\eqref{eqv}-\eqref{bc} as a model example. 
In the future, we plan to generalize our approach to other important classes of differential operators. We hope that Yurko's method will be useful for further extension of inverse problem theory, in particular, to non-self-adjoint operators with singular coefficients. 

It worth to mention that the inverse problem for equation~\eqref{eqv} with the Dirichlet boundary conditions $y(0) = y(\pi) = 0$ in the self-adjoint case was studied by Hryniv and Mykytyuk in \cite{HM-sd}. They also formulated the results for the case of the Robin boundary conditions~\eqref{bc} without detailed proofs. Anyway, the case of the Robin boundary conditions can be reduced to the case of the Dirichlet boundary conditions by the Darboux-type transformations. For the Sturm-Liouville operators with singular potentials, such transformations were recently obtained by Guliyev \cite{Gul19}. This is another way to solve our inverse problem in the self-adjoint case. Nevertheless, as far as we know, Inverse Problem~\ref{ip:1} in the non-self-adjoint case has not been studied before.

Throughout this paper, we follow the general strategy of Yurko's method from \cite[Sections 1.4, 1.6]{FY01} and focus on the differences caused by the singular potential in more detail. Section~2 provides preliminaries, in particular, transformation operators and asymptotic formulas for solutions of equation~\eqref{eqv}. Section~3 is devoted to the uniqueness theorem for Inverse Problem~\ref{ip:1}. In Section~4, we construct the main equation of the inverse problem in an appropriate Banach space and investigate the properties of the operator, participating in that equation. In Section~5, we formulate and prove our main theorems on necessary and sufficient conditions of the inverse problem solvability (Theorem~\ref{thm:nsc} for the self-adjoint case and Theorem~\ref{thm:nsc2} for the non-self-adjoint case). Section~5 contains also Algorithm~\ref{alg:1} for constructive solution of the inverse problem. In Appendix, an auxiliary lemma concerning entire functions is proved.

\section{Preliminaries}

First of all, let us introduce some {\bf notations}:

\begin{itemize}
\item $\la = \rho^2$, $\tau := \mbox{Im}\,\rho$, $\rho_n := \sqrt{\la_n}$, $\arg{\rho_n} \in [-\tfrac{\pi}{2}, \tfrac{\pi}{2})$, $n \ge 0$. 
\item The symbol $C$ denotes various positive constants independent of $x$, $n$, $\la$, etc.
\item The notation $\{ \varkappa_n \}$ is used for various sequences from $l_2$.
\item The notation $\varkappa_a(\rho)$ is used for various entire functions in $\rho$ of exponential type not greater than $a$ and such that $\int_{\mathbb R} |\varkappa_a(\rho)|^2 \, d\rho < \iy$. By Paley-Wiener Theorem,
$$
\varkappa_a(\rho) = \int_{-a}^a f(x) \exp(i \rho x) \, dx, \quad f \in L_2(-a, a).
$$
Further we need the estimate $\varkappa_a(\rho) = o(\exp(|\tau|a))$ as $|\rho| \to \iy$ uniformly by $\arg \rho$. 
\item $\langle z, y \rangle := z y^{[1]} - z^{[1]} y$.
\end{itemize}

The following theorem provides transformation operators for the solution $\vv(x, \la)$ and its quasi-derivative $\vv^{[1]}(x, \la)$. 

\begin{thm} \label{thm:trans}
The following relations hold
\begin{align} \label{transK}
    & \vv(x, \la) = \cos \rho x + \int_0^x \mathscr K(x, t) \cos \rho t \, dt, \\ \label{transN}
    & \vv^{[1]}(x, \la) = -\rho \sin \rho x + \rho \int_0^x \mathscr N(x, t) \sin \rho t \, dt + \mathscr C(x),
\end{align}
where the kernels $\mathscr K$ and $\mathscr N$ are square integrable in the region $\mathbb D := \{ (x, t) \colon 0 < t < x < \pi \}$ and the function $\mathscr C$ is continuous on $[0, \pi]$. Moreover, for each fixed $x \in (0, \pi]$, the functions $\mathscr K(x, .)$ and $\mathscr N(x, .)$ belong to $L_2(0, x)$ and the corresponding $L_2$-norms $\| \mathscr K(x, .) \|_{L_2(0, x)}$ and $\| \mathscr N(x, .) \|_{L_2(0, x)}$ are uniformly bounded with respect to $x \in (0, \pi]$. Analogously, for each fixed $t \in [0, \pi)$, the functions $\mathscr K(., t)$ and $\mathscr N(., t)$ belong to $L_2(t, \pi)$ and the corresponding $L_2$-norms are uniformly bounded with respect to $t \in [0, \pi)$.
\end{thm}

Note that the representation~\eqref{transK} for $\vv(x, \la)$ was obtained in \cite{HM-trans}. We prove Theorem~\ref{thm:trans}, since we especially need the formula~\eqref{transN}.

\begin{proof}
Using equation~\eqref{eqv}, one can easily show that the following system of Volterra integral equations is fulfilled:
\begin{multline} \label{voltphi}
    \vv(x, \la) = \cos \rho x + \int_0^x \left( \cos \rho (x - t) \sigma(t) - \frac{\sin \rho (x - t)}{\rho} \sigma^2(t)\right) \vv(t, \la) \, dt \\ - \int_0^x \frac{\sin \rho (x - t)}{\rho} \sigma(t) \vv^{[1]}(t, \la) \, dt, 
\end{multline}
\vspace*{-20pt}
\begin{multline} \label{voltphi1}
    \vv^{[1]}(x, \la) = -\rho \sin \rho x - \int_0^x (\rho \sin \rho(x - t) \sigma(t) + \cos \rho (x - t) \sigma^2(t)) \vv(t, \la) \, dt \\ - \int_0^x \cos \rho (x - t) \sigma(t) \vv^{[1]}(t, \la) \, dt. 
\end{multline}

Formal substitution of the representations~\eqref{transK}, \eqref{transN} into \eqref{voltphi}, \eqref{voltphi1} yields the following system of integral equations with respect to the triple $(\mathscr K, \mathscr N, \mathscr C)$:
\begin{equation} \label{intsys}
\left.
\begin{array}{rl}
    \mathscr K(x, t) & = \mathscr K_0(x, t) + I_{\mathscr K}(\mathscr K, \mathscr N, \mathscr C), \\
    \mathscr N(x, t) & = \mathscr N_0(x, t) + I_{\mathscr N}(\mathscr K, \mathscr N, \mathscr C), \\
    \mathscr C(x) & = \mathscr C_0(x) + I_{\mathscr C}(\mathscr K, \mathscr C),
\end{array}\right\}
\end{equation}
where
\begin{align*}
    \mathscr K_0(x, t) & = \tfrac{1}{2} \sigma\left( \tfrac{x + t}{2}\right) + \tfrac{1}{2} \sigma\left( \tfrac{x - t}{2}\right) - \tfrac{1}{2} \int_0^x \sigma^2(s) \, ds - \tfrac{1}{4} \int_t^x \left( \sigma^2(\tfrac{x - s}{2}) - \sigma^2(\tfrac{x + s}{2})\right) \, ds, \\
    \mathscr N_0(x, t) & = \tfrac{1}{2} \sigma\left( \tfrac{x + t}{2}\right) - \tfrac{1}{2} \sigma\left( \tfrac{x - t}{2}\right) + \tfrac{1}{2} \int_0^x \sigma^2(s) \, ds + \tfrac{1}{4} \int_t^x \left( \sigma^2(\tfrac{x - s}{2}) + \sigma^2(\tfrac{x + s}{2})\right) \, ds, \\
    \mathscr C_0(x) & = -\tfrac{1}{2} \int_0^x \sigma^2(t)\,dt - \tfrac{1}{4} \int_0^x (\sigma^2(\tfrac{x + t}{2}) + \sigma^2(\tfrac{x - t}{2})) \, dt, \\
    I_{\mathscr K}(\mathscr K, \mathscr N, \mathscr C) & = \frac{1}{2} \int_{x - t}^x \mathscr K(s, t - x + s) \sigma(s) \, ds + \frac{1}{2}\int_{\frac{x - t}{2}}^{x - t} \mathscr K(s, x - s - t) \sigma(s) \, ds \\ & + \frac{1}{2}\int_{\frac{x + t}{2}}^x \mathscr K(s, x - s + t) \sigma(s) \, ds - \frac{1}{2} \int_{t}^x d\xi \biggl( \int_{x - \xi}^x \mathscr K(s, \xi - x + s) \sigma^2(s) \, ds \\ & + \int_{\frac{x - \xi}{2}}^{x - \xi} \mathscr K(s, x - s - \xi) \sigma^2(s) \, ds - \int_{\frac{x + \xi}{2}}^x \mathscr K(s, x - s + \xi) \sigma^2(s) \, ds \biggr) \\ & + \frac{1}{2} \int_{x - t}^x \mathscr N(s, t - x + s) \sigma(s) \, ds - \frac{1}{2}\int_{\frac{x - t}{2}}^{x - t} \mathscr N(s, x - s - t) \sigma(s) \, ds \\ & - \frac{1}{2}\int_{\frac{x + t}{2}}^x \mathscr N(s, x - s + t) \sigma(s) \, ds - \int_0^{x - \xi} \mathscr C(s) \sigma(s) \, ds, \\
    I_{\mathscr N}(\mathscr K, \mathscr N, \mathscr C) & = -\frac{1}{2} \int_{x - t}^x \mathscr K(s, t - x + s) \sigma(s) \, ds - \frac{1}{2}\int_{\frac{x - t}{2}}^{x - t} \mathscr K(s, x - s - t) \sigma(s) \, ds \\ & + \frac{1}{2}\int_{\frac{x + t}{2}}^x \mathscr K(s, x - s + t) \sigma(s) \, ds + \frac{1}{2} \int_{t}^x d\xi \biggl( \int_{x - \xi}^x \mathscr K(s, \xi - x + s) \sigma^2(s) \, ds \\ & + \int_{\frac{x - \xi}{2}}^{x - \xi} \mathscr K(s, x - s - \xi) \sigma^2(s) \, ds + \int_{\frac{x + \xi}{2}}^x \mathscr K(s, x - s + \xi) \sigma^2(s) \, ds \biggr) \\ & - \frac{1}{2} \int_{x - t}^x \mathscr N(s, t - x + s) \sigma(s) \, ds + \frac{1}{2}\int_{\frac{x - t}{2}}^{x - t} \mathscr N(s, x - s - t) \sigma(s) \, ds \\ & - \frac{1}{2}\int_{\frac{x + t}{2}}^x \mathscr N(s, x - s + t) \sigma(s) \, ds + \int_0^{x - \xi} \mathscr C(s) \sigma(s) \, ds, \\
    I_{\mathscr C}(\mathscr K, \mathscr C) & = -\frac{1}{2} \int_0^x d\xi \biggl( \int_{x - \xi}^x \mathscr K(s, \xi - x + s) \sigma^2(s) \, ds + \int_{\frac{x - \xi}{2}}^{x - \xi} \mathscr K(s, x - s - \xi) \sigma^2(s) \, ds \\ & + \int_{\frac{x + \xi}{2}}^x \mathscr K(s, x - s + \xi) \sigma^2(s) \, ds \biggr) - \int_0^x \mathscr C(s) \sigma(s) \, ds.
\end{align*}

We solve the system~\eqref{intsys} by iterations:
$$
\mathscr K_{n + 1} := I_{\mathscr K}(\mathscr K_n, \mathscr N_n, \mathscr C_n), \quad
\mathscr N_{n + 1} := I_{\mathscr N}(\mathscr K_n, \mathscr N_n, \mathscr C_n), \quad
\mathscr C_{n + 1} := I_{\mathscr C}(\mathscr K_n, \mathscr C_n), \quad n \ge 0.
$$
Induction yields the estimates:
\begin{gather*}
|\mathscr K_n(x, t)|, |\mathscr N_n(x, t)| \le \tfrac{1}{2} \left( |\sigma(\tfrac{x + t}{2})| + |\sigma(\tfrac{x - t}{2})|\right) Q^n(x) \sqrt{\tfrac{x^n}{n!}} + a^n Q^{n + 1}(x) \sqrt{\tfrac{x^{n-1}}{(n-1)!}}, \\
|\mathscr C_n(x)| \le a^n Q^{n + 1}(x) \sqrt{\tfrac{x^n}{n!}}, \quad n \ge 1,
\end{gather*}
where $Q(x) = \| \sigma \|_{L_2(0, x)}$ and $a$ is a positive constant depending on $\| \sigma \|_{L_2(0, \pi)}$.
Consequently, the series
$$
\mathscr K := \sum_{n = 0}^{\iy} \mathscr K_n, \quad \mathscr N := \sum_{n = 0}^{\iy} \mathscr N_n
$$
converge in $L_2(\mathbb D)$, and the series
$$
\mathscr C(x) := \sum_{n = 0}^{\iy} \mathscr C_n(x)
$$
converges absolutely and uniformly on $[0, \pi]$.
By construction, the functions $\mathscr K$, $\mathscr N$, and $\mathscr C$ satisfy the relations~\eqref{transK}, \eqref{transN}, and all the other claimed properties.
\end{proof}

For each fixed $x \in [0, \pi]$, the relations~\eqref{transK} and~\eqref{transN} yield the asymptotic formulas
\begin{gather}
\label{asymptphi}
    \vv(x, \la) = \cos \rho x + \varkappa_x(\rho), \quad
    \vv^{[1]}(x, \la) = - \rho \sin \rho x + \rho \varkappa_x(\rho) + \vv^{[1]}(x, 0), \\    
    \label{asymptDelta}
    \Delta(\la) = - \rho \sin \rho \pi + \rho \varkappa_{\pi}(\rho) + \Delta(0).
\end{gather}

The relation~\eqref{asymptDelta} implies the standard estimate from below:
\begin{gather} \label{Dbelow}
|\Delta(\la)| \ge C_{\de}|\rho| \exp(|\tau| \pi), \quad \rho \in G_{\de}, \quad |\rho| \ge \rho^*, \\ \label{defG}
G_{\de} := \{ \rho \in \mathbb C \colon |\rho - n| \ge \de, \: n \in \mathbb Z \}.
\end{gather}
Here we mean that, for every $\de > 0$, there exist positive constants $\rho^*$ and $C_{\de}$ such that \eqref{Dbelow} holds.

Denote by $S(x, \la)$ and $\Psi(x, \la)$ the solutions of equation~\eqref{eqv} satisfying the initial conditions
$$
S(0, \la) = 0, \quad S^{[1]}(0, \la) = 1, \quad \Psi(\pi, \la) = 1, \quad \Psi^{[1]}(\pi, \la) = -H.
$$
For each fixed $x \in [0, \pi]$, the functions $\Psi(x, \la)$ and $S(x, \la)$ are entire in $\la$. The following asymptotic formulas can be obtained similarly to~\eqref{asymptphi}:
\begin{equation}  \label{asymptpsi}
    \Psi(x, \la) = \cos \rho (\pi - x) + \varkappa_{\pi-x}(\rho), \quad \Psi^{[1]}(x, \la) = \rho \sin \rho (\pi - x) + \rho \varkappa_{\pi - x}(\rho) + \Psi^{[1]}(x, 0).
\end{equation}

The {\it Weyl solution} $\Phi(x, \la)$ is the solution of equation~\eqref{eqv}, satisfying the boundary conditions $\Phi^{[1]}(0, \la) = 1$, $\Phi^{[1]}(\pi, \la) + H \Phi(\pi, \la) = 0$. The {\it Weyl function} is defined as follows: $M(\la) := \Psi(0, \la)$. One can easily obtain the relations
\begin{gather} \label{Phi1}
    \Phi(x, \la) = -\frac{\Psi(x, \la)}{\Delta(\la)}, \quad M(\la) = -\frac{\Psi(0, \la)}{\Delta(\la)}, \\ \label{Phi2}
    \Phi(x, \la) = S(x, \la) + M(\la) \vv(x, \la), \\ \label{wron}
    \langle \vv(x, \la), \Phi(x, \la) \rangle \equiv 1.
\end{gather}
Consequently, $\Phi(x, \la)$ for each fixed $x \in [0, \pi]$ and $M(\la)$ are meromorphic functions of $\la$ with simple poles at $\la = \la_n$, $n \ge 0$. Note that
\begin{equation} \label{resM}
\Res_{\la = \la_n} M(\la) = \al_n, \quad n \ge 0.
\end{equation}

\begin{lem} \label{lem:asympt}
The following asymptotic relations hold
\begin{equation} \label{asymptla}
    \rho_n = n + \varkappa_n, \quad \al_n = \frac{2}{\pi} + \varkappa_n, \quad n \ge 0.
\end{equation}
\end{lem}

\begin{proof}
The asymptotic formula for $\rho_n$ is obtained by the standard method, based on Rouch\'{e}'s Theorem and the relation~\eqref{asymptDelta}. In order to study the asymptotic behavior of $\al_n$, we combine~\eqref{Phi1} and~\eqref{resM}:
\begin{equation} \label{aln}
\al_n = -\frac{\Psi(0, \la_n)}{\tfrac{d}{d\la} \Delta(\la_n)}.
\end{equation}
In view of~\eqref{asymptDelta} and~\eqref{asymptpsi}, we have
$$
\Psi(0, \la) = \cos \rho \pi + \varkappa_{\pi}(\rho), \quad \tfrac{d}{d\la} \Delta(\la) = - \frac{\pi}{2} \cos \rho \pi + \varkappa_{\pi}(\rho).
$$
Putting $\rho_n = n + \varkappa_n$, we arrive at the relation $\al_n = \tfrac{2}{\pi} + \varkappa_n$.
\end{proof}

Spectral data asymptotics analogous to~\eqref{asymptla} for the case of the Dirichlet boundary conditions are provided in \cite{HM-sd, SS10}.

Similarly to \cite[Theorem~1.4.6]{FY01}, we obtain the formula
\begin{equation} \label{seriesM}
M(\la) = \sum_{n = 0}^{\iy} \frac{\al_n}{\la - \la_n}.
\end{equation}
Thus, the Weyl function is uniquely determined by the spectral data and vice versa. Therefore Inverse Problem~\ref{ip:1} is equivalent to the following one.

\begin{ip}
Given the Weyl function $M(\la)$, find $\sigma$ and $H$.
\end{ip}

\section{Uniqueness}

The goal of this section is to prove the uniqueness theorem for Inverse Problem~\ref{ip:1}, by using the method of spectral mappings.

Along with the problem $L = L(\sigma, H)$, we consider another problem $\tilde L = L(\tilde \sigma, \tilde H)$ of the same form as $L$, but with different coefficients $\tilde \sigma$ and $\tilde H$. We agree that, if a certain symbol $\ga$ denotes an object related to $L$, the symbol $\tilde \ga$ with tilde will denote the analogous object related to $\tilde L$. We emphasize the difference of quasi-derivatives for these two problems: $y^{[1]} = y' - \sigma y$ for $L$ and $y^{[1]} = y' - \tilde \sigma y$ for $\tilde L$. The following theorem says that the spectral data uniquely specify the problem $L(\sigma, H)$.

\begin{thm} \label{thm:uniq}
If $\la_n = \tilde \la_n$ and $\al_n = \tilde \al_n$ for all $n \ge 0$, then $\sigma = \tilde \sigma$ in $L_2(0, \pi)$ and $H = \tilde H$.
\end{thm}

\begin{proof}
Introduce the matrix of spectral mappings $P(x, \la) = [P_{jk}(x, \la)]_{j, k = 1, 2}$ as follows:
\begin{equation} \label{defP}
P(x, \la) \begin{bmatrix}
            \tilde \vv(x, \la) & \tilde \Phi(x, \la) \\
            \tilde \vv^{[1]}(x, \la) & \tilde \Phi^{[1]}(x, \la)
          \end{bmatrix} =
          \begin{bmatrix}
             \vv(x, \la) & \Phi(x, \la) \\
             \vv^{[1]}(x, \la) & \Phi^{[1]}(x, \la).
          \end{bmatrix}
\end{equation}
Using \eqref{defP} and \eqref{wron}, we derive the relations (the arguments $(x, \la)$ are omitted for brevity):
\begin{equation} \label{Pjk}
    P_{11} = \vv \tilde \Phi^{[1]} - \Phi \tilde \vv^{[1]}, \quad
    P_{12} = \Phi \tilde \vv - \vv \tilde \Phi.
\end{equation}
Consequently,
\begin{equation} \label{sm1}
    P_{11} = 1 + \vv (\tilde \Phi^{[1]} - \Phi^{[1]}) - \Phi (\tilde \vv^{[1]} - \vv^{[1]}), \quad
    P_{12} = \Phi(\tilde \vv - \vv) - \vv (\tilde \Phi - \Phi).
\end{equation}
It follows from \eqref{asymptphi}-\eqref{Dbelow} and~\eqref{Phi1} that, for each fixed $x \in [0, \pi]$,
\begin{gather*}
\left.
\begin{array}{c}
\vv(x, \la) = O(\exp(|\tau|x)), \quad \tilde \vv(x, \la) - \vv(x, \la) = o(\exp(|\tau|x)) \\
\tilde \vv^{[1]}(x, \la) - \vv^{[1]}(x, \la) = o(\rho\exp(|\tau|x))
\end{array}\right\} \quad \rho \in \mathbb C,\\
\left.
\begin{array}{c}
\Phi(x, \la) = O(\rho^{-1} \exp(-|\tau|x)),  \quad
\tilde \Phi(x, \la) - \Phi(x, \la) = o(\rho^{-1}\exp(-|\tau|x)) \\
\tilde \Phi^{[1]}(x, \la) - \Phi^{[1]}(x, \la) = o(\exp(-|\tau|x))
\end{array} \right\} \quad \rho \in G_{\de},
\end{gather*}
as $|\rho| \to \iy$ uniformly by $\arg \rho$.
Substituting these estimates into~\eqref{sm1}, we obtain
\begin{equation} \label{estP}
P_{1j}(x, \la) - \de_{1j} = o(1), \quad \rho \in G_{\de}, \quad |\rho| \to \iy, 
\end{equation}
where $j = 1, 2$, $\de_{jk}$ is the Kronecker delta.

Furthermore, according to~\eqref{Pjk} and~\eqref{Phi2}, we get
\begin{align*}
& P_{11} = \vv \tilde S^{[1]} - S \tilde \vv^{[1]} + (\tilde M - M) \vv \tilde \vv^{[1]}, \\
& P_{12} = S \tilde \vv - \vv \tilde S + (M - \tilde M) \vv \tilde \vv.
\end{align*}

In view of~\eqref{seriesM}, the conditions of the theorem imply $M(\la) \equiv \tilde M(\la)$. Consequently, the functions $P_{1j}(x, \la)$ are entire in $\la$-plane for each fixed $x \in [0, \pi]$, $j = 1, 2$. Using~\eqref{estP} and Liouville's Theorem, we conclude that $P_{11}(x, \la) \equiv 1$, $P_{12}(x, \la) \equiv 0$. Then the relation~\eqref{defP} yields $\vv(x, \la) \equiv \tilde \vv(x, \la)$.

Subtracting the relations
$$
\begin{array}{l}
-(\vv' - \sigma \vv)' - \sigma (\vv' - \sigma \vv) - \sigma^2 \vv = \la \vv, \\
-(\vv' - \tilde \sigma \vv)' - \tilde \sigma (\vv' - \tilde \sigma \vv) - \tilde \sigma^2 \vv = \la \vv, \\
\end{array}
$$
we obtain the relation
$$
((\sigma - \tilde \sigma) \vv)' = (\sigma - \tilde \sigma) \vv',
$$
which holds a.e. on $(0, \pi)$. In addition, the function $(\sigma - \tilde \sigma) \vv$ is absolutely continuous on $[0, \pi]$. Fix such $\la$ that $\vv(x, \la) \ne 0$, $x \in [0, \pi]$. Then it is clear that the function $(\sigma - \tilde \sigma)$ is absolutely continuous on $[0, \pi]$ and $(\sigma - \tilde \sigma)' = 0$ a.e. on $(0, \pi)$. Thus $\sigma - \tilde \sigma \equiv C$. Using the boundary conditions $\vv^{[1]}(0, \la) = 0$ and $\tilde \vv^{[1]}(0, \la) = 0$, we conclude that $\sigma(0) - \tilde \sigma(0) = 0$, so $\sigma(x) = \tilde \sigma(x)$ a.e. on $(0, \pi)$. Hence the quasi-derivatives $y^{[1]}$ coincide for $L$ and $\tilde L$. Comparing the boundary conditions of $L$ and $\tilde L$ at $x = \pi$, we conclude that $H = \tilde H$.
\end{proof}

\section{Main equation in a Banach space}

In this section, Inverse Problem~\ref{ip:1} is reduced to a linear equation in a Banach space. First, we use the contour integration in the $\la$-plane to derive an infinite system of linear equations (see Lemma~\ref{lem:cont}). Second, a special Banach space is introduced, and the system is rewritten as the so-called main equation in that space. Third, we investigate the properties of the operator participating in the main equation. Those properties play an important role in the next section.

Suppose that we have two boundary value problems $L = L(\sigma, H)$ and $\tilde L = L(\tilde \sigma, \tilde H)$. Everywhere below, we assume that $\tilde \sigma(x) = H = 0$, $x \in [0, \pi]$. Define the function
\begin{equation} \label{defD}
\tilde D(x, \la, \mu) := \frac{\langle \tilde \vv(x, \la), \tilde \vv(x, \mu) \rangle}{\la - \mu} = \int_0^x \tilde \vv(t, \la) \tilde \vv(x, \mu) \, dt.
\end{equation}

Introduce the notations
\begin{gather*}
\la_{n0} := \la_n, \quad \la_{n1} := \tilde \la_n, \quad \rho_{n0} := \rho_n, \quad \rho_{n1} := \tilde \rho_n, \quad \al_{n0} := \al_n, \quad \al_{n1} := \tilde \al_n, \\
\vv_{ni}(x) := \vv(x, \la_{ni}), \quad \tilde \vv_{ni}(x) := \tilde \vv(x, \la_{ni}), \quad i = 0, 1, \\
\xi_n := |\rho_n - \tilde \rho_n| + |\al_n - \tilde \al_n|, \quad n \ge 0.
\end{gather*}

In \cite[Section~1.6.1]{FY01}, the following estimates have been obtained:
\begin{gather} \label{estD1}
    |\tilde D(x, \la, \la_{kj})| \le \frac{C \exp(|\tau| x)}{|\rho - k| + 1}, \quad |\tilde D(x, \la, \la_{k0}) - \tilde D(x, \la, \la_{k1})| \le \frac{C \exp(|\tau| x) \xi_k}{|\rho - k| + 1}, \\
    \label{estD2}
    |\tilde D(x, \la_{ni}, \la_{kj})| \le \frac{C}{|n - k| + 1}, \quad |\tilde D(x, \la_{ni}, \la_{k0}) - \tilde D(x, \la_{ni}, \la_{k1})| \le \frac{C \xi_k}{|n - k| + 1},
\end{gather}
where $x \in [0, \pi]$, $\mbox{Re}\, \rho \ge 0$, $n, k \ge 0$, $i, j = 0, 1$.

\begin{lem} \label{lem:cont}
The following relation holds
\begin{equation} \label{relphi}
    \tilde \vv(x, \la) = \vv(x, \la) + \sum_{k = 0}^{\iy} (\al_{k0} \tilde D(x, \la, \la_{k0}) \vv_{k0}(x) - \al_{k1} \tilde D(x, \la, \la_{k1}) \vv_{k1}(x))
\end{equation}
where the series converges absolutely and uniformly by $x \in [0, \pi]$ and by $\la$ on any compact set.
\end{lem}

\begin{proof}
Consider the contour $\mathscr S(\tau) := (-\iy + i\tau, +\iy + i \tau)$ in the $\rho$-plane, $\tau > 0$. Denote by $\Upsilon$ the contour in the $\la$-plane which the image of $\mathscr S(\tau)$ under the mapping $\la = \rho^2$. Let $\Xi$ be the image of the half-plane $\{ \mbox{Im} \rho > \tau \}$ under this mapping. We choose $\tau > 0$ such that $\la_{nj} \in \Xi$ for all $n \ge 0$, $j = 0, 1$. In view of the asymptotics~\eqref{asymptla}, such $\tau$ always exists.
Define the region
$$
C_N := \{ \la \in \mathbb C \colon |\la| < (N + 1/2)^2 \}, \quad N \in \mathbb N.
$$
Denote by $\Upsilon_N$ and $\Gamma_N$ the boundaries of the regions $\Xi \cap C_N$ and $C_N$, respectively (with the counter-clockwise circuit). 

Repeating the arguments of the proof of \cite[Lemma~1.6.3]{FY01}, we obtain the relation
\begin{equation} \label{sm2}
\vv(x, \la) = \tilde \vv(x, \la) + \frac{1}{2 \pi i} \int_{\Upsilon_N} \frac{\tilde \vv(x, \la) P_{11}(x, \xi) + \tilde \vv^{[1]}(x, \la) P_{12}(x, \xi)}{\la - \xi}\, d\xi + \eps_N(x, \la),
\end{equation}
where
$$
\eps_N(x, \la) = -\frac{1}{2 \pi i} \int_{\Gamma_N} \frac{\tilde \vv(x, \la) (P_{11}(x, \xi) - 1) + \tilde \vv^{[1]}(x, \la) P_{12}(x, \xi)}{\la - \xi} \, d\xi.
$$
Using~\eqref{estP}, we show that
$\lim_{N \to \iy} \eps_N(x, \la) = 0$ uniformly by $x \in [0, \pi]$ and $\la$ on compact sets. Substituting~\eqref{Pjk} into~\eqref{sm2}, we obtain
\begin{align*}
\vv(x, \la) = \tilde \vv(x, \la) + \frac{1}{2 \pi i} \int_{\Upsilon_N} & \Bigl(\vv(x, \la) (\vv(x, \xi) \tilde \Phi^{[1]}(x, \xi) - \Phi(x, \xi) \tilde \vv^{[1]}(x, \xi)) \\ & + \tilde \vv^{[1]}(x, \la) (\Phi(x, \xi) \tilde \vv(x, \xi) - \vv(x, \xi) \tilde \Phi(x, \xi))\Bigr) \frac{d\xi}{\la - \xi} + \eps_N(x, \la)
\end{align*}
Using~\eqref{Phi2} and \eqref{defD}, we derive the relation
$$
\tilde \vv(x, \la) = \vv(x, \la) + \frac{1}{2 \pi i} \int_{\Upsilon_N} \tilde D(x, \la, \xi) (M(\xi) - \tilde M(\xi)) \vv(x, \xi) \, d\xi + \eps_N(x, \la),
$$
because the terms with $S(x, \xi)$ and $\tilde S(x, \xi)$ vanish by Cauchy Theorem. Calculating the integral by Residue Theorem and passing to the limit as $N \to \iy$, we arrive at \eqref{relphi}.
\end{proof}

Define the set of indices $J := \{ (n, i) \colon n \ge 0, \, i = 0, 1 \}$ and the functions
\begin{equation} \label{defR}
\tilde R_{ni, kj}(x) := (-1)^j \al_{kj} \tilde D(x, \la_{ni}, \la_{kj}), \quad (n, i), (k, j) \in J.
\end{equation}
Setting $\la = \la_{ni}$ in~\eqref{relphi}, we obtain the following system of linear equations with respect to $\{ \vv_{ni} \}_{(n, i) \in J}$:
\begin{equation} \label{main}
\tilde \vv_{ni}(x) = \vv_{ni}(x) + \sum_{k = 0}^{\iy} \sum_{j = 0}^1 \tilde R_{ni, kj}(x) \vv_{kj}(x), \quad (n, i) \in J.
\end{equation}
Note that the series in \eqref{main} converges only in the sense $\lim_{N \to \iy} \sum_{k = 0}^N (\dots)$. Therefore, further we transform~\eqref{main} into a linear equation in a special Banach space.

Let $B$ be the Banach space of infinite sequences in the form $f = \{ f_{ni} \}_{(n, i) \in J}$, satisfying the conditions: 

\smallskip

(i) If $\la_{n0} = \la_{n1}$, then $f_{n0} = f_{n1}$.

\smallskip

(ii) $\| f \|_B := \sup_{n \ge 0} \max \{ |f_{n0}|, \chi_n |f_{n0} - f_{n1}| \} < \iy$, where 
$$
\chi_n = \begin{cases}
            |\rho_{n0} - \rho_{n1}|^{-1}, \quad \text{if} \, \rho_{n0} \ne \rho_{n1}, \\
            0, \quad \text{otherwise}.
        \end{cases}
$$

\smallskip

For simplicity, we assume that $\la_{n0} \ne \la_{k1}$ for $n \ne k$. One can easily achieve this condition by a shift of the spectrum.

In view of the asymptotic formulas~\eqref{asymptla} and~\eqref{asymptphi}, for each fixed $x \in [0, \pi]$, the sequences $\phi(x) := \{ \vv_{ni}(x)\}_{(n, i) \in J}$ and $\tilde \phi(x) := \{ \tilde \vv_{ni}(x)\}_{(n, i) \in J}$ belong to $B$. For each fixed $x \in [0, \pi]$, we define the linear operator $\tilde R(x) \colon B \to B$, acting on an element $f \in B$ by the following rule:
\begin{equation} \label{opR}
(\tilde R(x) f)_{ni} := \sum_{k = 0}^{\iy} \sum_{j = 0}^1 \tilde R_{ni, kj}(x) f_{kj}, \quad (n, i) \in J.
\end{equation}

\begin{lem} \label{lem:RB}
For each fixed $x \in [0, \pi]$, the operator $\tilde R(x)$ is bounded and can be approximated by finite-dimensional operators with respect to the operator norm $\| . \|_{B \to B}$.
\end{lem}

\begin{proof}
Using~\eqref{defR} and~\eqref{opR}, we obtain
\begin{align*}
(\tilde R(x) f)_{ni} & = \sum_{k = 0}^{\iy} (\al_{k0} \tilde D(x, \la_{ni}, \la_{k0}) f_{k0} - \al_{k1} \tilde D(x, \la_{ni}, \la_{k1}) f_{k1}) \\ & =
\sum_{k = 0}^{\iy} \bigl( (\al_{k0} - \al_{k1}) \tilde D(x, \la_{ni}, \la_{k0}) f_{k0} + \al_{k1} (\tilde D(x, \la_{ni}, \la_{k0}) - \tilde D(x, \la_{ni}, \la_{k1})) f_{k0} \\ & + \al_{k1} \tilde D(x, \la_{ni}, \la_{k1}) (f_{k0} - f_{k1}) \bigr). 
\end{align*}
Since $f \in B$, we have
\begin{equation} \label{estf}
    |f_{k0}| \le \| f \|_B, \quad |f_{k0} - f_{k1}| \le \| f \|_B \xi_k, \quad k \ge 0.
\end{equation}
Using \eqref{asymptla}, \eqref{estD2}, and~\eqref{estf}, we conclude that 
\begin{equation} \label{sm3}
|(\tilde R(x) f)_{ni}| \le C \| f \|_B \sum_{k = 0}^{\iy} \frac{\xi_k}{|n - k| + 1}, \quad (n, i) \in J.
\end{equation}
Analogously, one can show that
\begin{equation} \label{sm4}
    |(\tilde R(x) f)_{n0} - (\tilde R(x) f)_{n1}| \le C \| f \|_B |\rho_{n0} - \rho_{n1}| \sum_{k = 0}^{\iy} \frac{\xi_k}{|n - k| + 1}, \quad n \ge 0.
\end{equation}
Combining~\eqref{sm3} and~\eqref{sm4} and using the definition of $\| . \|_B$, we arrive at the important estimate
\begin{equation} \label{estR}
\| \tilde R(x) \|_{B \to B} \le C \sup_{n \ge 0}\sum_{k = 0}^{\iy} \frac{\xi_k}{|n - k| + 1}. 
\end{equation}
Since $\{ \xi_k \} \in l_2$, we get
$$
\| \tilde R(x) \|_{B \to B} \le C  \left( \sum_{k = 0}^{\iy} \xi_k^2 \right)^{1/2}.
$$
Hence, the operator $\tilde R(x)$ is bounded from $B$ to $B$ uniformly by $x$ in $[0, \pi]$.

In addition, the operator $\tilde R(x)$ is approximated by the finite-dimensional operators $\tilde R^N(x)$, $N \to \iy$, acting by the rule:
$$
(\tilde R^N(x) f)_{ni} = \sum_{k = 0}^N \sum_{j = 0}^1 \tilde R_{ni, kj}(x) f_{kj}, \quad f \in B, \quad (n, i) \in J. 
$$
\end{proof}

Note that our definition of the operator $\tilde R(x)$ is slightly different from the definition of the operator $\tilde H(x)$ in \cite[Section~1.6.1]{FY01}. The authors of \cite{FY01} define their operator in the standard space $m$ of bounded sequences but by more complicated formulas than ours. These two approaches are equivalent, but our approach is more convenient for future generalizations to matrix Sturm-Liouville operators.

Thus, for each fixed $x \in [0, \pi]$, the relations~\eqref{main} can be considered as the following linear equation in the Banach space $B$:
\begin{equation} \label{main2}
    (I + \tilde R(x)) \phi(x) = \tilde \phi(x),
\end{equation}
where $I$ is the identity operator in $B$. We call~\eqref{main2} the \textit{main equation} of Inverse Problem~\ref{ip:1}. Using this equation, one can constructively solve the inverse problem. Indeed, $\tilde \phi(x)$ and $\tilde R(x)$ can be constructed by the spectral data $\{ \la_n, \al_n \}_{n = 0}^{\iy}$ and the model problem $\tilde L$. The solution $\phi(x)$ of the main equation is related with $\sigma$ and $H$. We describe the algorithm for solving Inverse Problem~\ref{ip:1} in more detail in Section~5. Now we study some important properties of the operator $\tilde R(x)$.

\begin{lem} \label{lem:Rl2}
For each fixed $x \in [0, \pi]$, the operator $\tilde R(x)$ maps $B$ into $l_2$ and is bounded from $B$ to $l_2$ uniformly by $x \in [0, \pi]$.
\end{lem}

\begin{proof}
Using \eqref{defD}, \eqref{defR}, and~\eqref{opR}, we derive
\begin{equation} \label{sm5}
g_{ni}(x) := (\tilde R(x) f)_{ni} = \sum_{k = 0}^{\iy} \int_0^x \cos (\rho_{ni} t) (f_{k0} \cos (\rho_{k0} t) - f_{k1} \cos (\rho_{k1} t)) \, dt, \quad (n, i) \in J.
\end{equation}
By virtue of the estimates~\eqref{asymptla} and~\eqref{estf}, the series
$$
F(t) := \sum_{k = 0}^{\iy} (f_{k0} \cos(\rho_{k0} t) - f_{k1} \cos(\rho_{k1} t))
$$
converges in $L_2(0, \pi)$ and 
\begin{equation} \label{estFf}
\| F \|_{L_2(0, \pi)} \le C \| f \|_B.
\end{equation}
Consequently, one can swap the sum and the integral in~\eqref{sm5} and arrive at the relation
\begin{equation} \label{sm6}
g_{ni}(x) = \int_0^x F(t) \cos(\rho_{ni} t) \, dt.
\end{equation}
In view of \eqref{asymptla}, the sequences $\{ \cos (\rho_{ni} t) \}_{n \ge 0}$ for $i = 0, 1$ are Riesz bases in $L_2(0, \pi)$ (see \cite{HV01}). The formula~\eqref{sm6} gives the Fourier coefficients of the functions $F_{[0, x]}(t)$ with respect to the corresponding biorthonormal bases, where
$$
F_{[0, x]}(t) = \begin{cases}
                F(t), \quad t \in (0, x), \\
                0, \quad t \in (x, \pi).
            \end{cases}
$$
Hence, $\{ g_{ni}(x) \}_{(n, i) \in J} \in l_2$ and
$$
\| \{ g_{ni}(x) \} \|_{l_2} \le C \| F_{[0, x]} \|_{L_2(0, \pi)} \le C \| F \|_{L_2(0, \pi)} \le C \| f \|_B.
$$
This yields the claim.
\end{proof}

\begin{lem} \label{lem:Rl1}
Suppose that $f = \{ f_{ni} \}_{(n, i)\in J} \in B$ and $g_{ni}(x) := (\tilde R(x) f)_{ni}$. Then the sequence $\{ (g_{n0} - g_{n1})(x) \}_{n = 0}^{\iy}$ belongs to $l_1$ for each fixed $x \in [0, \pi]$ and
$$
\| \{ (g_{n0} - g_{n1})(x) \} \|_{l_1} \le C \| f \|_{l_2},
$$
where the constant $C$ does not depend on $x$.
\end{lem}

\begin{proof}
Using~\eqref{asymptla}, we obtain the relation
\begin{equation} \label{difcos}
\cos (\rho_{n0} t) - \cos (\rho_{n1} t) = \cos ((n + \varkappa_n)t) - \cos (n t)
= - \varkappa_n t \sin (n t) + O(\varkappa_n^2),
\end{equation}
where the $O$-estimate is uniform with respect to $t \in [0, \pi]$. Substituting~\eqref{difcos} into~\eqref{sm6}, we get
\begin{gather*}
g_{n0}(x) - g_{n1}(x) = \int_0^x F(t) (\cos (\rho_{n0} t) - \cos (\rho_{n1} t)) \, dt = k_n(x) \varkappa_n + r_n(x), \\
k_n(x) := -\int_0^x F(t) t \sin (n t) \, dt, \quad |r_n(x)| \le C \varkappa_n^2 \| F \|_{L_2(0, \pi)}.
\end{gather*}
For the Fourier coefficients $\{ k_n(x) \}$, we have
$$
\left( \sum_{n = 0}^{\iy} |k_n(x)|^2\right)^{1/2} \le C \| F \|_{L_2(0, \pi)}, \quad x \in [0, \pi].
$$
Consequently,
$$
\sum_{n = 0}^{\iy} |g_{n0}(x) - g_{n1}(x)| \le C \| F \|_{L_2(0, \pi)}.
$$
The latter estimate together with~\eqref{estFf} yield the assertion of the lemma.
\end{proof}

\begin{lem} \label{lem:Rcont}
$\tilde R(x)$ is continuous with respect to $x \in [0, \pi]$ in the space of linear bounded operators from $B$ to $B$.
\end{lem}

\begin{proof}
Let $x_0$ and $x$ be arbitrary points in $[0, \pi]$. Following the proof of Lemma~\ref{lem:RB}, we show that
$$
\| \tilde R(x) - \tilde R(x_0) \|_{B \to B} \le C |x - x_0| \sum_{k = 1}^N \xi_k + C \sup_{n \ge 0} \sum_{k = N + 1}^{\iy} \frac{\xi_k}{|n - k| + 1}
$$
for every $N \in \mathbb N$. For the second term, we have
$$
T_N := C \sup_{n \ge 0} \sum_{k = N + 1}^{\iy} \frac{\xi_k}{|n - k| + 1} \le C_1 \left( \sum_{k = N + 1}^{\iy} \xi_k^2 \right)^{1/2}.
$$
Hence, for each fixed $\eps > 0$, one can choose $N \in \mathbb N$ such that $T_N < \frac{\eps}{2}$. Fix this value of $N$ and $x_0 \in [0, \pi]$. Then one can find sufficiently small $\de > 0$ such that, for every $x \in [0, \pi]$ satisfying $|x - x_0| \le \de$, the following estimate holds
$$
C |x - x_0| \sum_{k = 1}^N \xi_k \le \frac{\eps}{2}.
$$
Thus, for each fixed $x_0 \in [0, \pi]$ and $\eps > 0$, there exists $\de > 0$ such that, for every $x \in [0, \pi]$ satisfying $|x - x_0| \le \de$, we have $\| \tilde R(x) - \tilde R(x_0) \|_{B \to B} \le \eps$, i.e., $\tilde R(x)$ is continuous at every $x_0 \in [0, \pi]$.
\end{proof}

\section{Necessary and sufficient conditions}

In this section, we provide necessary and sufficient conditions of solvability for Inverse Problem~\ref{ip:1}. First, we consider the self-adjoint case, when the function $\sigma(x)$ is real-valued and the constant $H$ is real. For this case, we formulate the necessary and sufficient conditions in Theorem~\ref{thm:nsc} and prove the sufficiency part of that theorem. An essential part of our proof is Lemma~\ref{lem:solve}, which asserts the unique solvability of the main equation~\eqref{main2}. Further, we develop the constructive Algorithm~\ref{alg:1} for solving Inverse Problem~\ref{ip:1}. In the end of this section, necessary and sufficient conditions are obtained for the general non-self-adjoint case (see Theorem~\ref{thm:nsc2}).

\begin{thm} \label{thm:nsc}
For numbers $\{ \la_n, \al_n \}_{n = 0}^{\iy}$ to be the spectral data of a boundary value problem $L(\sigma, H)$ in the self-adjoint case (i.e., $\sigma(x)$, $x \in (0, \pi)$, and $H$ are real), it is necessary and sufficient to fulfill the following conditions:

\smallskip

(i) $\la_n$, $\al_n$ are real, $\al_n > 0$ and $\la_n \ne \la_k$ for all $n, k \ge 0$, $n \ge k$.

\smallskip

(ii) The asymptotic formulas~\eqref{asymptla} hold.
\end{thm}

The necessity part of Theorem~\ref{thm:nsc} is already known, so we focus on the proof of sufficiency. Let $\{ \la_n, \al_n \}_{n = 0}^{\iy}$ be arbitrary numbers satisfying the conditions (i), (ii) of Theorem~\ref{thm:nsc}. Using $\{ \la_n, \al_n \}_{n = 0}^{\iy}$ and the model problem $\tilde L = L(0, 0)$, we construct the Banach space $B$, the element $\tilde \phi(x) \in B$, and the operator $\tilde R(x) \colon B \to B$ for each $x \in [0, \pi]$, as it was described in the previous section. Consider the equation~\eqref{main2} with respect to the unknown element $\phi(x)$.

\begin{lem} \label{lem:solve}
For each fixed $x \in [0, \pi]$, the operator $\tilde R(x) \colon B \to B$ has a bounded inverse, so the main equation~\eqref{main2} has a unique solution $\tilde \phi(x) \in B$.
\end{lem}

\begin{proof}
Let $x \in [0, \pi]$ be fixed. By Lemma~\ref{lem:RB}, the operator $\tilde R(x)$ can be approximated by finite-dimensional operators. Therefore, according to Fredholm's Theorem, it is sufficient to prove that the corresponding homogeneous equation
\begin{equation} \label{homo}
(I + \tilde R(x)) \beta(x) = 0, \quad \beta(x) = \{ \beta_{ni}(x) \}_{(n, i) \in J} \in B,
\end{equation}
has the only solution $\beta(x) = 0$ in $B$. Since $\beta(x) = -\tilde R(x) \beta(x)$, Lemmas~\ref{lem:Rl2} and~\ref{lem:Rl1} imply that 
\begin{equation} \label{beta}
\{ \beta_{ni}(x) \}_{(n, i) \in J} \in l_2, \quad \{ \beta_{n0}(x) - \beta_{n1}(x) \}_{n = 0}^{\iy} \in l_1.
\end{equation}

Define the functions
\begin{gather} \label{defga}
    \gamma(x, \la) := -\sum_{k = 0}^{\iy} (\al_{k0} \tilde D(x, \la, \la_{k0}) \beta_{k0}(x) - \al_{k1} \tilde D(x, \la, \la_{k1}) \beta_{k1}(x)), \\ \label{defGa}
    \Gamma(x, \la) := -\sum_{k = 0}^{\iy} (\al_{k0} \tilde E(x, \la, \la_{k0}) \beta_{k0}(x) - \al_{k1} \tilde E(x, \la, \la_{k1}) \beta_{k1}(x)), \\ \label{defE}
    \tilde E(x, \la, \mu) := \frac{\langle \tilde \Phi(x, \la), \tilde \vv(x, \mu)\rangle}{\la - \mu}, \\ \nonumber
    \mathscr B(x, \la) := \overline{\gamma(x, \overline{\la})} \Gamma(x, \la).
\end{gather}

The function $\ga(x, \la)$ is entire in the $\la$-plane, the functions $\Gamma(x, \la)$ and $\mathscr B(x, \la)$ are meromorphic in $\la$ with the simple poles $\{ \la_{ni} \}$. In view of \eqref{homo}, $\ga(x, \la_{ni}) = \be_{ni}$, $(n, i) \in J$. It is easy to check that
\begin{equation} \label{ResB}
\Res_{\la = \la_{n0}} \mathscr B(x, \la) = \al_{n0} |\be_{n0}(x)|^2, \quad \Res_{\la = \la_{n1}} \mathscr B(x, \la) = 0 \quad (\text{if} \: \la_{n0} \ne \la_{n1}).
\end{equation}

Using~\eqref{defga}, we derive
\begin{align*}
    \ga(x, \la) = & -\sum_{k = 0}^{\iy} (\al_{k0} - \al_{k1}) \tilde D(x, \la, \la_{k0}) \be_{k0}(x) - \sum_{k = 0}^{\iy} \al_{k1} (\tilde D(x, \la, \la_{k0}) - \tilde D(x, \la, \la_{k1})) \be_{k0}(x) \\ & - \sum_{k = 0}^{\iy} \al_{k1} \tilde D(x, \la, \la_{k1}) (\be_{k0}(x) - \be_{k1}(x)).
\end{align*}
Using~\eqref{asymptla}, \eqref{estD1}, and~\eqref{beta}, we get
\begin{equation} \label{estga}
|\ga(x, \la)| \le C(x) \exp(|\tau|x) \sum_{k = 0}^{\iy} \frac{\theta_k}{|\rho - k| + 1}.
\end{equation}
Here and below, $\rho = \sqrt{\la}$, $\mbox{Re}\, \rho \ge 0$, and the notation $\{ \theta_k \}$ is used for various sequences from $l_1$. Analogously, using~\eqref{defGa} and~\eqref{defE}, we obtain the estimate
$$
|\Gamma(x, \la)| \le \frac{C(x)}{|\rho|} \exp(-|\tau|x) \sum_{k = 0}^{\iy} \frac{\theta_k}{|\rho - k| + 1}, \quad \rho \in G_{\de}, \quad |\rho| \ge \rho^*,
$$
for sufficiently large $\rho^*$ and sufficiently small $\de > 0$.
Hence,
$$
|\mathscr B(x, \la)| \le \frac{C(x)}{|\rho|} \left( \sum_{k = 0}^{\iy} \frac{\theta_k}{|\rho - k| + 1}\right)^2 \le \frac{C(x)}{|\rho|} \sum_{k = 0}^{\iy} (\sqrt \theta_k)^2 \cdot \sum_{k = 0}^{\iy} \frac{(\sqrt{\theta_k})^2}{(|\rho - k| + 1)^2}, \quad \rho \in G_{\de}, \: |\rho| \ge \rho^*.
$$
Suppose that $\la \in \Gamma_N$, $\Gamma_N = \{ \la \colon |\la| = (N + 1/2)^2 \}$, $N \in \mathbb N$ is sufficiently large. Then
$$
|\mathscr B(x, \la)| \le \frac{C(x) f_N}{N}, \quad f_N := \sum_{k = 0}^{\iy} \frac{\theta_k}{(N + 1/2 - k)^2}.
$$
Obviously, $\{ f_N \} \in l_1$. This implies
$$
\varliminf_{N \to \iy} \frac{f_N}{1/N} = 0.
$$
Thus there exists a sequence $\{ N_k \}$ such that $\mathscr B(x, \la) = o(N_k^{-2})$ as $k \to \iy$ uniformly by $\la \in \Gamma_{N_k}$. Consequently,
$$
\lim_{k \to \iy} \int_{\Gamma_{N_k}} \mathscr B(x, \la) \, d\la = 0.
$$
Calculating the integral by the Residue Theorem and using~\eqref{ResB}, we arrive at the relation
$$
\lim_{k \to \iy} \sum_{n = 0}^{N_k} \al_{n0} |\be_{n0}(x)|^2 = 0.
$$
Since $\al_{n0} > 0$, we get $\be_{n0}(x) = 0$ for all $n \ge 0$.

Consider the entire function
\begin{equation} \label{prodD}
\Delta(\la) := \pi (\la_0 - \la) \prod_{n = 1}^{\iy} \frac{\la_n - \la}{n^2}.
\end{equation}
It follows from the relation $\ga(x, \la_n) = \beta_{n0}(x) = 0$, $n \ge 0$, that the function $\frac{\ga(x, \la)}{\Delta(\la)}$ is entire. In addition, \eqref{asymptla} and \eqref{prodD} imply the estimate~\eqref{Dbelow}. The estimate~\eqref{estga} yields that $\ga(x, \la) = O(\exp(|\tau|x))$. Consequently,
$$
\frac{\ga(x, \la)}{\Delta(\la)} = O(\rho^{-1}), \quad |\rho| \to \iy.
$$
By virtue of Liouville's Theorem, $\ga(x, \la) \equiv 0$. Hence, $\be_{n1}(x) = \ga(x, \la_{n1}) = 0$, $n \ge 0$. Thus we have shown that the homogeneous equation~\eqref{homo} has the only solution $\be(x) = 0$, so the lemma is proved.
\end{proof}

\begin{lem} \label{lem:psi}
Let $\phi(x) = \{ \vv_{ni}(x)\}_{(n, i) \in J}$ be the solution of the main equation~\eqref{main2}. Then its elements can be represented in the form $\vv_{ni}(x) = \cos (n x) + \psi_{ni}(x)$, $(n, i) \in J$, where the functions $\psi_{ni}(x)$ are continuous on $[0, \pi]$, the sequence $\{ \psi_{ni}(x) \}_{(n, i) \in J}$ belongs to $l_2$ for each fixed $x \in [0, \pi]$, and $\| \{ \psi_{ni}(x)\} \|_{l_2}$ is uniformly bounded for $x \in [0, \pi]$. Moreover, the series
\begin{equation} \label{defTheta}
\Theta(x) := \sum_{n = 0}^{\iy} (\psi_{n0} - \psi_{n1})(x) \cos (n x)
\end{equation}
converges in $L_2(0, \pi)$, and $\{ (\psi_{n0} - \psi_{n1})(\pi) \}_{n = 0}^{\iy} \in l_1$.
\end{lem}

\begin{proof}
By Lemma~\ref{lem:solve}, there exists the operator $\tilde P(x) = (I + \tilde R(x))^{-1}$, bounded for each fixed $x \in [0, \pi]$. In particular, $\| P(x_0) \| < \iy$. Here and below $\| . \| = \| . \|_{B \to B}$.
By Lemma~\ref{lem:Rcont}, $\tilde R(x)$ is continuous at $x = x_0$, so there exists $\de > 0$ such that, for every $x \in [0, \pi]$ satisfying $|x - x_0| \le \de$, the following estimate holds:
$$
\| \tilde R(x_0) - \tilde R(x) \| \le \frac{1}{2 \| P(x_0) \|}.
$$
Using~\cite[Lemma~1.5.1]{FY01}, we obtain that
$$
\tilde P(x) - \tilde P(x_0) = \sum_{k = 1}^{\iy} (\tilde R(x_0) - \tilde R(x))^k (\tilde P(x_0))^{k + 1}.
$$
Consequently,
$$
\| \tilde P(x) - \tilde P(x_0) \| \le 2 \| \tilde P(x_0) \|^2 \| \tilde R(x_0) - \tilde R(x) \| \to 0, \quad x \to x_0.
$$
Thus $\tilde P(x)$ is continuous in $x$ in the space of linear bounded operators from $B$ to $B$, so $\| \tilde P(x) \|_{B \to B}$ is uniformly bounded for $x \in [0, \pi]$. Since $\phi(x) = \tilde P(x) \tilde \phi(x)$, we get that $\phi(x)$ is continuous by $x$ in $B$ and $\| \phi(x) \|_B$ is uniformly bounded for $x \in [0, \pi]$.

Denote
\begin{gather} \label{defpsi}
\psi_{ni}(x) := \vv_{ni}(x) - \cos (n x), \quad
\tilde \psi_{ni}(x) := \tilde \vv_{ni}(x) - \cos (n x), \\ \nonumber
\psi(x) := \{ \psi_{ni}(x) \}_{(n, i) \in J}, \quad
\tilde \psi(x) := \{ \tilde \psi_{ni}(x) \}_{(n, i) \in J}.
\end{gather}
The main equation~\eqref{main2} yields
\begin{equation} \label{relpsi}
\psi(x) = \tilde \psi(x) - \tilde R(x) \phi(x).
\end{equation}
Using the asymptotic formulas~\eqref{asymptla}, we get that $\tilde \psi(x) \in l_2$ and $\| \tilde \psi(x) \|_{l_2}$ is bounded uniformly by $x \in [0, \pi]$. Together with Lemma~\ref{lem:Rl2}, this yields that $\psi(x) \in l_2$ and $\| \psi(x) \|_{l_2}$ is uniformly bounded by $x \in [0, \pi]$.

According to Lemma~\ref{lem:Rl1}, we have
\begin{equation} \label{sm7}
\{ (\tilde R(x) \phi(x))_{n0} - (\tilde R(x) \phi(x))_{n1} \}_{n = 0}^{\iy} \in l_1.
\end{equation}
Using~\eqref{difcos}, we get
$$
\tilde \psi_{n0}(x) - \tilde \psi_{n1}(x) = \cos((n + \varkappa_n) x) - \cos (n x) = -\varkappa_n x \sin (n x) + O(\varkappa_n^2).
$$
Consider the series
\begin{equation} \label{deftTheta}
\tilde \Theta(x) := \sum_{n = 0}^{\iy} (\tilde \psi_{n0} - \tilde \psi_{n1})(x) \cos (n x) = -\frac{x}{2} \sum_{n = 0}^{\iy} \varkappa_n \sin (2 n x) + \sum_{n = 0}^{\iy} O(\varkappa_n^2).
\end{equation}
Obviously, the series $\tilde \Theta(x)$ converges in $L_2(0, \pi)$. Combining \eqref{relpsi}, \eqref{sm7} and~\eqref{deftTheta}, we conclude that the series \eqref{defTheta} converges in $L_2(0, \pi)$ and $\{ (\psi_{n0} - \psi_{n1})(\pi) \}_{n = 0}^{\iy} \in l_1$.
\end{proof}

Construct the function $\sigma(x)$ and the real $H$ as follows:
\begin{gather} \label{defsi}
    \sigma(x) := -2 \sum_{n = 0}^{\iy} \bigl(\al_{n0} \vv_{n0}(x) \tilde \vv_{n0}(x) - \al_{n1} \vv_{n1}(x) \tilde \vv_{n1}(x) - \tfrac{1}{2}(\al_{n0} - \al_{n1})\bigr), \\ \label{defH}
    H := - \sum_{n = 0}^{\iy} (\al_{n0} \vv_{n0}(\pi) \tilde \vv_{n0}(\pi) - \al_{n1} \vv_{n1}(\pi) \tilde \vv_{n1}(\pi) - (\al_{n0} - \al_{n1})).
\end{gather}

\begin{lem} \label{lem:L2}
The series \eqref{defsi} converges in $L_2(0, \pi)$ and the series~\eqref{defH} converges.
\end{lem}

\begin{proof}
Substituting~\eqref{defpsi} into~\eqref{defsi}, we represent $\sigma(x)$ in the form
\begin{gather*}
\sigma(x) = -2 (Z_1(x) + Z_2(x) + Z_3(x)), \\
Z_1(x) := \sum_{n = 0}^{\iy} (\al_{n0} - \al_{n1}) \bigl(\cos^2 (nx) - \tfrac{1}{2}\bigr) = \sum_{n = 0}^{\iy} \varkappa_n \cos (2 n x), \\
Z_2(x) := \frac{2}{\pi} \sum_{n = 0}^{\iy} (\psi_{n0}(x) + \tilde \psi_{n0}(x) - \psi_{n1}(x) - \tilde \psi_{n1}(x)) \cos (n x) = \frac{2}{\pi} (\Theta(x) + \tilde \Theta(x)), \\
Z_3(x) := \sum_{n = 0}^{\iy} (\al_{n0} - \al_{n1}) (\psi_{n0}(x) + \tilde \psi_{n0}(x)) \cos (n x) + \sum_{n = 0}^{\iy} (\al_{n0} \psi_{n0}(x) \tilde \psi_{n0}(x) - \al_{n1} \psi_{n1}(x) \tilde \psi_{n1}(x))
\end{gather*}
Obviously, $Z_1(x)$ converges in $L_2(0, \pi)$. The convergence of $Z_2(x)$ in $L_2(0, \pi)$ follows from Lemma~\ref{lem:psi}. Lemma~\ref{lem:psi} also implies that the sequences $\{ \psi_{ni}(x) \}$ and $\{ \tilde \psi_{ni}(x) \}$ belong to $l_2$, and their elements are continuous on $[0, \pi]$.
Therefore the series $Z_3(x)$ converge absolutely and uniformly on $[0, \pi]$. Hence, $\sigma \in L_2(0, \pi)$.

Substituting~\eqref{defpsi} into~\eqref{defH}, we get
\begin{align*}
H & = \frac{2 (-1)^{n + 1}}{\pi} \sum_{n = 0}^{\iy} (\psi_{n0}(\pi) + \tilde \psi_{n0}(\pi) - \psi_{n1}(\pi) - \tilde \psi_{n1}(\pi)) \\ & +
(-1)^{n + 1} \sum_{n = 0}^{\iy} (\al_{n0} - \al_{n1}) (\psi_{n0}(\pi) + \tilde \psi_{n0}(\pi)) - \sum_{n = 0}^{\iy} (\al_{n0} \psi_{n0}(\pi) \tilde \psi_{n0}(\pi) - \al_{n1} \psi_{n1}(\pi) \tilde \psi_{n1}(\pi))
\end{align*}
By virtue of Lemma~\ref{lem:psi}, we have $\psi(\pi) \in l_2$ and $\{ (\psi_{n0} - \psi_{n1})(\pi) \} \in l_1$, and the same relations are valid for $\tilde \psi(\pi)$. Using~\eqref{asymptla} together with Lemma~\ref{lem:psi}, we conclude that the series for $H$ converges.
\end{proof}

Consider the problem $L = L(\sigma, H)$, where $\sigma$ and $H$ are constructed by~\eqref{defsi} and~\eqref{defH}, respectively.

\begin{lem} \label{lem:sd}
The numbers $\{ \la_n, \al_n \}_{n = 0}^{\iy}$ coincide with the spectral data of $L(\sigma, H)$.
\end{lem}

In order to prove Lemma~\ref{lem:sd}, along with $\{ \la_n, \al_n \}_{n = 0}^{\iy}$, we consider the data $\{ \la_n^N, \al_n^N \}_{n = 0}^{\iy}$ defined as follows:
\begin{equation} \label{deflaN}
\la_n^N = \begin{cases}
            \la_n, \quad n \le N, \\
            \tilde \la_n, \quad n > N,
          \end{cases}    \quad
\al_n^N = \begin{cases}
            \al_n, \quad n \le N, \\
            \tilde \al_n, \quad n > N,
          \end{cases}    
\quad N \in \mathbb N.          
\end{equation}

For each fixed $N \in \mathbb N$, the data $\{ \la_n, \al_n \}_{n = 0}^{\iy}$ satisfy necessary and sufficient conditions for being the spectral data of the Sturm-Liouville problem with a regular potential (see \cite[Theorem~1.6.2]{FY01}). Hence, there exist a real-valued function $q^N \in L_2(0, \pi)$ and reals $h^N$, $g^N$ such that $\{ \la_n^N, \al_n^N \}_{n = 0}^{\iy}$ are the spectral data of the boundary value problem
\begin{gather*}
-y'' + q^N(x) y = \la y, \quad x \in (0, \pi), \\
y'(0) - h^N y(0) = 0, \quad y'(\pi) + g^N y(\pi) = 0.
\end{gather*}
The latter problem is equivalent to $L(\sigma^N, H^N)$ with 
\begin{equation} \label{sm8}
\sigma^N(x) := h^N + \int_0^x q^N(t) \, dt, \quad H^N := g^N + \sigma^N(\pi).
\end{equation}
By virtue of \cite[Lemma~1.6.5]{FY01}, the data $q^N$, $h^N$ and $g^N$ can be constructed by the formulas
\begin{gather} \label{sm9}
    \eps_0^N(x) := \sum_{n = 0}^N (\al_{n0} \vv_{n0}^N(x) \tilde \vv_{n0}(x) - \al_{n1} \vv_{n1}^N(x) \tilde \vv_{n1}(x)), \\ \label{sm10}
    q^N(x) = -2 \frac{d}{dx} \eps_0^N(x), \quad h^N = -\eps_0^N(0), \quad g^N = \eps_0^N(\pi).
\end{gather}
Here $\phi^N(x) = \{ \vv_{ni}^N(x) \}_{(n, i) \in J}$ is the solution of the main equation
\begin{equation} \label{mainN}
(I + \tilde R^N(x)) \phi^N(x) = \tilde \phi^N(x),
\end{equation}
analogous to~\eqref{main2}. The operator $\tilde R^N(x)$ and the infinite vector $\tilde \phi^N(x)$ are constructed similarly to $\tilde R(x)$ and $\tilde \phi(x)$, respectively, by using the data $\{ \la_n^N, \al_n^N \}_{n = 0}^{\iy}$ and the model problem $\tilde L = L(0, 0)$. Note that $\tilde \vv_{ni}^N(x) = \tilde \vv_{ni}(x)$ for all $n \le N$, $i = 0, 1$, and $\vv_{n0}^N(x) = \vv_{n1}^N(x)$, $\tilde \vv_{n0}^N(x) = \tilde \vv_{n1}^N(x)$ for all $n > N$.

Using~\eqref{sm8}-\eqref{sm10}, we derive the formulas
\begin{gather} \label{defsiN}
    \sigma^N(x) = -2 \sum_{n = 0}^N \bigl( \al_{n0} \vv_{n0}^N(x) \tilde \vv_{n0}(x) - \al_{n1} \vv_{n1}^N(x) \tilde \vv_{n1}(x) - \tfrac{1}{2}(\al_{n0} - \al_{n1})\bigr), \\ \label{defHN}
    H^N = -\sum_{n = 0}^N (\al_{n0} \vv_{n0}^N(\pi) \tilde \vv_{n0}(\pi) - \al_{n1} \vv_{n1}^N(\pi) \tilde \vv_{n1}(\pi) - (\al_{n0} - \al_{n1})).
\end{gather}

\begin{lem} \label{lem:limN}
$\sigma^N \to \sigma$ in $L_2(0, \pi)$ and $H^N \to H$ as $N \to \iy$, where $\sigma$, $H$, $\sigma^N$, $H^N$ are defined by \eqref{defsi}, \eqref{defH}, \eqref{defsiN}, \eqref{defHN}, respectively.
\end{lem}

\begin{proof}
According to the main equations~\eqref{main2} and~\eqref{mainN}, the following relations hold for $n \le N$, $i = 0, 1$, $x \in [0, \pi]$:
\begin{align*}
    & \vv_{ni}(x) + \sum_{k = 0}^{\iy} \sum_{j = 0}^1 \tilde R_{ni, kj}(x) \vv_{kj}(x) = \tilde \vv_{ni}(x), \\
    & \vv_{ni}^N(x) + \sum_{k = 0}^N \sum_{j = 0}^1 \tilde R_{ni, kj}(x) \vv_{kj}^N(x) = \tilde \vv_{ni}(x).
\end{align*}
Subtraction yields
$$
(\vv_{ni}(x) - \vv_{ni}^N(x)) + \sum_{k = 0}^N \sum_{j = 0}^1 \tilde R_{ni, kj}(x) (\vv_{kj}(x) - \vv_{kj}^N(x)) + \sum_{k = N + 1}^{\iy} \sum_{j = 0}^1 \tilde R_{ni, kj}(x) \vv_{kj}(x) = 0.
$$
Define $\theta^N(x) = \{ \theta_{ni}^N(x)\}_{(n, i) \in J}$ as follows:
$$
\theta_{ni}^N(x) := \begin{cases}
                        \vv_{ni}(x) - \vv_{ni}^N(x), \quad n \le N, \\
                        0, \quad n > N.
                   \end{cases}    
$$
For each fixed $x \in [0, \pi]$ and $N \in \mathbb N$, the sequence $\theta^N(x)$ belongs to $B$ and satisfies the equation
\begin{equation} \label{eqthe}
(I + \tilde R(x)) \theta^N(x) = \zeta^N(x),
\end{equation}
where $\zeta^N(x) = \{ \zeta_{ni}^N(x) \}_{(n, i) \in J} \in B$,
$$
\zeta_{ni}^N(x) := \begin{cases} -\sum_{k = N + 1}^{\iy} \tilde R_{ni, kj}(x) \vv_{kj}(x), \quad n \le N, \\ 0, \quad n > N. \end{cases}
$$
Analogously to the proof of Lemma~\ref{lem:RB}, we obtain the estimate:
$$
\| \zeta^N(x) \|_B \le C \| \vv(x) \|_B \sup_{n \ge 0} \sum_{k = N + 1}^{\iy} \frac{\xi_k}{|n - k| + 1} \to 0, \quad N \to \iy,
$$
uniformly by $x \in [0, \pi]$.

According to Lemmas~\ref{lem:solve} and~\ref{lem:psi}, the operator $(I + \tilde R(x))^{-1}$ exists and is bounded uniformly by $x \in [0, \pi]$. Consequently, the solution $\theta^N(x) = (I + \tilde R(x))^{-1} \zeta^N(x)$ of equation~\eqref{eqthe} tends to zero in $B$ as $N \to \iy$ uniformly by $x \in [0, \pi]$. Using Lemma~\ref{lem:Rl2}, we get $\tilde R(x) \theta^N(x) \in l_2$ and $\| \tilde R(x) \theta^N(x) \|_{l_2} \to 0$, $\| \zeta^N(x) \|_{l_2} \to 0$ as $N \to \iy$ uniformly by $x \in [0, \pi]$. Using~\eqref{eqthe}, we conclude that $\theta^N(x) \in l_2$ and $\| \theta^N(x) \|_{l_2} \to 0$ as $N \to \iy$ uniformly by $x \in [0, \pi]$. Now relying on Lemma~\ref{lem:Rl1} and its proof, we obtain
\begin{equation} \label{limthe}
\lim_{N \to \iy} \sum_{n = 0}^N |\theta_{n0}^N(x) - \theta_{n1}^N(x)| = 0
\end{equation}
uniformly by $x \in [0, \pi]$.

Subtracting~\eqref{defsiN} from \eqref{defsi}, we get
\begin{gather*}
    \sigma(x) - \sigma^N(x) = -2(\mathscr S_1^N(x) + \mathscr S_2^N(x)), \\
    \mathscr S_1^N(x) := \sum_{n = 0}^N (\al_{n0} (\vv_{n0}(x) - \vv_{n0}^N(x)) \tilde \vv_{n0}(x) - \al_{n1}(\vv_{n1}(x) - \vv_{n1}^N(x)) \tilde \vv_{n1}(x)), \\
    \mathscr S_2^N(x) := \sum_{n = N + 1}^{\iy} (\al_{n0} \vv_{n0}(x) \tilde \vv_{n0}(x) - \al_{n1} \vv_{n1}(x) \tilde \vv_{n1}(x)).
\end{gather*}
By Lemma~\ref{lem:L2}, the series for $\sigma(x)$ converges in $L_2(0, \pi)$, so $\| \mathscr S_2^N \|_{L_2} \to 0$ as $N \to \iy$. In order to prove the same for $\mathscr S_1^N$, we derive
\begin{multline} \label{sm17}
\mathscr S_1^N(x) = \sum_{n = 0}^N (\al_{n0} - \al_{n1}) \theta_{n0}^N(x) \tilde \vv_{n0}(x) + \frac{2}{\pi}\sum_{n = 0}^N \theta_{n0}^N(x) (\tilde \vv_{n0}(x) - \tilde \vv_{n1}(x)) \\ + \frac{2}{\pi} \sum_{n = 0}^N (\theta_{n0}^N(x) - \theta_{n1}^N(x)) \cos nx.
\end{multline}
Recall that $\{ \al_{n0} - \al_{n1} \} \in l_2$, $\| \theta^N(x) \|_{l_2} \to 0$ as $N \to \iy$, $\tilde \vv_{n0}(x) = O(1)$, $\{ \tilde \vv_{n0}(x) - \tilde \vv_{n1}(x) \} \in l_2$ and~\eqref{limthe} holds. Therefore, all the three sums in~\eqref{sm17} tend to zero as $N \to \iy$ uniformly by $x \in [0, \pi]$. Thus $\sigma^N \to \sigma$ in $L_2(0, \pi)$ as $N \to \iy$. Analogously, we show that $H^N \to H$ as $N \to \iy$, by subtracting~\eqref{defHN} from~\eqref{defH}.
\end{proof}

\begin{lem} \label{lem:stab}
Suppose that $\sigma$ and $\sigma^N$, $N \in \mathbb N$, are arbitrary functions from $L_2(0, \pi)$ such that $\sigma^N \to \sigma$ in $L_2(0, \pi)$ as $N \to \iy$ and $H$, $H^N$, $N \in \mathbb N$, are arbitrary reals such that $H^N \to H$ as $N \to \iy$. Let $\{ \la_n, \al_n \}_{n = 0}^{\iy}$ and $\{ \la_n^N, \al_n^N \}_{n = 0}^{\iy}$ be the spectral data of the problems $L(\sigma, H)$ and $L(\sigma^N, H^N)$, respectively. Then
\begin{equation} \label{stab}
\lim_{N \to \iy} \sum_{n = 0}^{\iy} (|\rho_n^N - \rho_n| + |\al_n^N - \al_n|)^2 = 0.
\end{equation}
\end{lem}

\begin{proof}
Theorem~\ref{thm:trans} implies that
\begin{equation} \label{DeltaPD}
\Delta(\la) = - \rho \sin \rho \pi + \rho \int_0^{\pi} \mathscr P(t) \sin \rho t \, dt + \mathscr D, \quad \mathscr P \in L_2(0, \pi),
\end{equation}
where
$$
\mathscr P(t) := \mathscr N(\pi, t) - H \left( 1 + \int_t^{\pi} \mathscr K(\pi, s) \, ds\right), \quad
\mathscr D := \mathscr C(\pi) + H\left( 1 + \int_0^{\pi} \mathscr K(\pi, s) \, ds\right).
$$
Similarly, the characteristic function $\Delta^N(\la)$ of the problem $L(\sigma^N, H^N)$ can be represented in the form
$$
\Delta^N(\la) = - \rho \sin \rho \pi + \rho \int_0^{\pi} \mathscr P^N(t) \sin \rho t \, dt + \mathscr D^N, \quad \mathscr P^N \in L_2(0, \pi).
$$

Analyzing the proof of Theorem~\ref{thm:trans}, we obtain that, if $\sigma^N \to \sigma$ in $L_2(0, \pi)$ and $H^N \to H$ as $N \to \iy$, then $\mathscr P^N \to \mathscr P$ in $L_2(0, \pi)$ and $\mathscr D^N \to \mathscr D$ as $N \to \iy$. By virtue of Lemma~\ref{lem:append}, we get
\begin{equation} \label{limrho}
\lim_{N \to \iy} \sum_{n = 0}^{\iy} |\rho_n^N - \rho_n|^2 = 0.
\end{equation}

It remains to prove the similar relation for $\al_n$. Using~\eqref{aln}, we obtain 
\begin{equation} \label{difal}
\al_n - \al_n^N = \frac{\tfrac{d}{d\la} \Delta(\la_n) (\Psi^N(0, \la_n^N) - \Psi(0, \la_n)) + (\tfrac{d}{d\la} \Delta(\la_n) - \tfrac{d}{d\la} \Delta^N(\la_n^N))\Psi(0, \la_n)}{\tfrac{d}{d\la} \Delta(\la_n) \tfrac{d}{d\la} \Delta^N(\la_n^N)}.
\end{equation}
Differentiating~\eqref{DeltaPD}, we get
\begin{equation} \label{difD}
\frac{d}{d\la} \Delta(\la) = -\frac{\pi}{2} \cos \rho \pi - \frac{1}{2\rho} \sin \rho \pi + \frac{1}{2} \int_0^{\pi} t \mathscr P(t) \cos \rho t \, dt + \frac{1}{2\rho} \int_0^{\pi} \mathscr P(t) \sin \rho t \, dt.
\end{equation}
Analogously to~\eqref{transK}, we have
\begin{equation} \label{transPsi}
\Psi(0, \la) = \cos \rho \pi + \int_0^{\pi} \mathscr R(t) \cos \rho t \, dt, \quad \mathscr R \in L_2(0, \pi).
\end{equation}
The relations similar to~\eqref{difD}, \eqref{transPsi} are valid for the functions $\tfrac{d}{d\la} \Delta^N(\la)$, $\Psi^N(0, \la)$ with $\mathscr P^N$, $\mathscr R^N$ instead of $\mathscr P$, $\mathscr R$, respectively. If $\sigma^N \to \sigma$ in $L_2(0, \pi)$ as $N \to \iy$, we have $\mathscr R^N \to \mathscr R$ in $L_2(0, \pi)$.

For sufficiently large $N_0$ and $N \ge N_0$, the relations~\eqref{asymptla}, \eqref{limrho}, \eqref{difD}, and~\eqref{transPsi} imply the estimates
\begin{equation} \label{estDPsi}
\left.
\begin{array}{c}
    |\tfrac{d}{d\la} \Delta(\la_n)|^{-1} \le C, \quad 
    |\tfrac{d}{d\la} \Delta^N(\la_n^N)|^{-1} \le C, \quad 
    |\tfrac{d}{d\la} \Delta(\la_n)| \le C, \quad
    |\Psi(0, \la_n)| \le C, \\ 
    |\Psi^N(0, \la_n^N) - \Psi(0, \la_n)| \le C (|\rho_n - \rho_n^N| + |\hat r_n^N|), \\
    |\tfrac{d}{d\la}\Delta(\la_n) - \tfrac{d}{d\la} \Delta^N(\la_n^N)| \le C (|\rho_n - \rho_n^N| + |\hat p_n^N| + |\hat d_n^N|),
\end{array}\right\}
\end{equation}
where
\begin{gather*}
\hat r_n^N = \int_0^{\pi} (\mathscr R(t) - \mathscr R^N(t)) \cos n t \, dt, \quad
\hat p_n^N = \int_0^{\pi} t (\mathscr P(t) - \mathscr P^N(t)) \cos nt \, dt, \\
\hat d_n^N = \int_0^{\pi} (\mathscr P(t) - \mathscr P^N(t)) \sin nt \, dt.
\end{gather*}
Recall that $\mathscr P^N \to \mathscr P$ and $\mathscr R^N \to \mathscr R$ in $L_2(0, \pi)$ as $N \to \iy$. Using Bessel's Inequality for the Fourier coefficients, we get
\begin{equation} \label{limP}
\lim_{N \to \iy} \sum_{n = 0}^{\iy} (|\hat r_n^N|^2 + |\hat p_n^N|^2 + |\hat d_n^N|^2) = 0.
\end{equation}
Combining~\eqref{limrho}, \eqref{difal}, \eqref{estDPsi} and~\eqref{limP}, we arrive at~\eqref{stab}.
\end{proof}

In view of the asymptotics~\eqref{asymptla}, the relation~\eqref{stab} holds for the data $\{ \la_n^N, \al_n^N \}_{n = 0}^{\iy}$, $N \in \mathbb N$, defined by~\eqref{deflaN}. By virtue of Lemma~\ref{lem:limN}, $\sigma^N \to \sigma$ in $L_2(0, \pi)$ and $H^N \to H$ as $N \to \iy$. Therefore, by Lemma~\ref{lem:stab}, the spectral data of $L(\sigma^N, H^N)$ converge to the spectral data of $L(\sigma, H)$ in the sense \eqref{stab}. Hence, the initially given numbers $\{ \la_n, \al_n \}_{n = 0}^{\iy}$ are the spectral data of the problem $L$, so Lemma~\ref{lem:sd} is proved. Lemmas~\ref{lem:solve}-\ref{lem:sd} together yield the sufficiency part of Theorem~\ref{thm:nsc}.

Finally, we arrive at Algorithm~\ref{alg:1} for constructive solution of Inverse Problem~\ref{ip:1}.

\begin{alg} \label{alg:1}
Let the data $\{ \la_n, \al_n \}_{n = 0}^{\iy}$ be given. We need to find $\sigma$ and $H$.

\begin{enumerate}
    \item Take the model problem $\tilde L = L(0, 0)$.
    \item Construct the sequence $\tilde \phi(x) = \{ \tilde \vv_{ni}(x) \}_{(n, i) \in J}$ and the operator $\tilde R(x)$, using~\eqref{defD}, \eqref{defR}.
    \item Solve the main equation~\eqref{main2} and so get $\phi(x) = \{ \vv_{ni}(x) \}_{(n, i) \in J}$.
    \item Find $\sigma$ and $H$ by the formulas~\eqref{defsi} and~\eqref{defH}, respectively.
\end{enumerate}
\end{alg}

Now we proceed to the non-self-adjoint case, when $\sigma(x)$, $x \in (0, \pi)$, and $H$ are not necessarily real. In this case, the proof of Lemma~\ref{lem:solve} fails. Therefore, we include the requirement of the unique solvability of the main equation into the necessary and sufficient conditions, given by the following theorem. We say that the problem $L(\sigma, H)$ belongs to the class $\mathcal V$ if $\la_n \ne \la_k$ for all $n \ne k$.

\begin{thm} \label{thm:nsc2}
For numbers $\{ \la_n, \al_n \}_{n = 0}^{\iy}$ to be the spectral data of a boundary value problem $L(\sigma, H) \in \mathcal V$, it is necessary and sufficient to fulfill the following conditions:

\smallskip

(i) $\la_n \ne \la_k$, $\al_n \ne 0$ for all $n, k \ge 0$, $n \ge k$.

\smallskip

(ii) The asymptotic formulas~\eqref{asymptla} hold.

\smallskip

(iii) For each fixed $x \in [0, \pi]$, the operator $(I + \tilde R(x)) \colon B \to B$ has the bounded inverse.
\end{thm}

Theorem~\ref{thm:nsc2} is analogous to \cite[Theorem~1.6.3]{FY01} for the case of a regular potential. By necessity, condition (iii) is proved similarly to \cite[Theorem~1.6.1]{FY01}. The sufficiency part follows from Lemmas~\ref{lem:psi}-\ref{lem:sd}, which are also valid for the non-self-adjoint case.

\section{Appendix}

In this section, we prove the auxiliary proposition (Lemma~\ref{lem:append}), which is used in the proof of Lemma~\ref{lem:stab}. Note that analogous propositions for the case of regular potentials have been obtained and applied to inverse problems in \cite{BB17} (see Lemma~3) and in \cite{YB20} (see Lemma~5).

Consider arbitrary entire functions of the form
\begin{align} \label{Delta}
    & \Delta(\la) = -\rho \sin \rho \pi + \rho \int_0^{\pi} \mathscr P(t) \sin \rho t \, dt + \mathscr D, \\ \label{Deltat}
    & \tilde \Delta(\la) = -\rho \sin \rho \pi + \rho \int_0^{\pi} \tilde {\mathscr P}(t) \sin \rho t \, dt + \tilde {\mathscr D},
\end{align}
where $\la = \rho^2$, the functions $\mathscr P$, $\tilde {\mathscr P}$ belong to $L_2(0, \pi)$, and $\mathscr D$, $\tilde {\mathscr D}$ are constants. Rouch\'e's Theorem implies that $\Delta(\la)$ and $\tilde \Delta(\la)$ have the zeros $\{ \rho_n^2 \}_{n = 0}^{\iy}$ and $\{ \tilde \rho_n^2 \}_{n = 0}^{\iy}$, respectively, $\arg \rho_n, \arg \tilde \rho_n \in [-\tfrac{\pi}{2}, \tfrac{\pi}{2})$ and
\begin{equation} \label{asymptrho}
\rho_n = n + \varkappa_n, \quad \tilde \rho_n = n + \tilde \varkappa_n, \quad
\{ \varkappa_n \}, \{ \tilde \varkappa_n \} \in l_2.
\end{equation}

\begin{lem} \label{lem:append}
Let $\Delta(\la)$ be a function of the form~\eqref{Delta} with simple zeros $\{ \rho_n^2 \}_{n = 0}^{\iy}$. Then there exists $\de > 0$ (depending on $\Delta(\la)$) such that, for any function $\tilde {\mathscr P} \in L_2(0, \pi)$ and any constant $\tilde {\mathscr D}$ satisfying the estimates 
\begin{equation} \label{estPD}
\| \mathscr P - \tilde {\mathscr P} \|_{L_2(0, \pi)} \le \de, \quad 
|\mathscr D - \tilde {\mathscr D}| \le \de,
\end{equation}
the zeros $\{ \tilde \rho_n^2 \}_{n = 0}^{\iy}$ of the function $\tilde \Delta(\la)$ defined by~\eqref{Deltat} satisfy the estimate
\begin{equation} \label{sumrho}
\left( \sum_{n = 0}^{\iy} |\rho_n - \tilde \rho_n|^2 \right)^{1/2} \le 
C (\| \mathscr P - \tilde {\mathscr P} \|_{L_2(0, \pi)} +
|\mathscr D - \tilde {\mathscr D}|),
\end{equation}
where the constant $C$ depends only on $\Delta(\la)$ and not on $\tilde {\mathscr P}$, $\tilde {\mathscr D}$.
\end{lem}

\begin{proof}
Consider the contours $\ga_{n,r} := \{ \rho \in \mathbb C \colon |\rho - \rho_n| = r \}$, where the radius $r > 0$ is fixed and so small that the contours $\{ \ga_{n,r} \}_{n \ge 0}$ do not intersect with each other and $(-\rho_m) \not\in \overline{\mbox{int}\, \ga_n}$ for all $n, m \ge 0$. (The case $\rho_n = 0$ requires minor changes). Define $s(\rho) := \Delta(\rho^2)$, $\tilde s(\rho) := \tilde \Delta(\rho^2)$. If $\tilde {\mathscr P}$ and $\tilde {\mathscr D}$ satisfy the estimates~\eqref{estPD} for some $\de > 0$, then 
$$
|s(\rho)| \ge \frac{C_0}{n + 1}, \quad |s(\rho) - \tilde s(\rho)| \le C (n + 1) \de, \quad \rho \in \ga_{n, r}, \quad n \ge 0,
$$
where the constants $C_0$ and $C$ depend on $r$ and not on $n$, $\rho$, $\tilde {\mathscr P}$, $\tilde {\mathscr D}$. Consequently, we can choose a sufficiently small $\de > 0$ such that
$$
\frac{|s(\rho) - \tilde s(\rho)|}{|s(\rho)|} < 1, \quad \rho \in \ga_{n, r}, \quad n \ge 0.
$$
By Rouch\'e's Theorem, we conclude that, for each $n \ge 0$, the function $\tilde s(\rho)$ has exactly one simple zero $\tilde \rho_n$ inside $\ga_{n, r}$.

The Taylor formula yields
\begin{equation} \label{sm11}
s(\tilde \rho_n) = s(\rho) + \tfrac{d}{d\rho} s(\theta_n) (\tilde \rho_n - \rho_n), \quad \theta_n \in \mbox{int} \, \ga_{n,r}, \quad n \ge 0.
\end{equation}
Subtracting~\eqref{Deltat} from~\eqref{Delta}, we get
\begin{equation} \label{sm12}
s(\tilde \rho_n) - \tilde s(\tilde \rho_n) = \tilde \rho_n \int_0^{\pi} \hat {\mathscr P}(t) \sin \tilde \rho_n t \, dt + \hat{\mathscr D}, \quad \hat{\mathscr P} := \mathscr P - \tilde{\mathscr P}, \quad \hat{\mathscr D} := \mathscr D - \tilde {\mathscr D}.
\end{equation}
Combining~\eqref{sm11} and~\eqref{sm12}, we arrive at the relation
\begin{equation} \label{sm13}
\tilde \rho_n - \rho_n = \bigl(\tfrac{d}{d\rho} s(\theta_n)\bigr)^{-1} \left( \tilde \rho_n \int_0^{\pi} \hat {\mathscr P}(t) \sin \tilde \rho_n t \, dt + \hat{\mathscr D} \right).
\end{equation}
Using~\eqref{Delta} and~\eqref{asymptrho}, we obtain the estimate
\begin{equation} \label{sm14}
\bigl| \tfrac{d}{d\rho}s(\rho)\bigr| \ge C (n + 1), \quad \rho \in \mbox{int}\,\ga_{n, r}, \quad n \ge 0.
\end{equation}
We also represent $\sin \tilde \rho_n t$ in the form
\begin{equation} \label{sm15}
\sin \tilde \rho_n t = \sin n t + \varkappa_n(t), \quad \left\{ \max_{t \in [0, \pi]} |\varkappa_n(t)| \right\}_{n \ge 0} \in l_2.
\end{equation}
Substituting~\eqref{sm14} and~\eqref{sm15} into~\eqref{sm13}, we get
\begin{equation} \label{sm16}
|\tilde \rho_n - \rho_n| \le C \left( |\hat p_n| + \| \hat {\mathscr P} \|_{L_2(0, \pi)} \varkappa_n + \frac{|\hat {\mathscr D}|}{n + 1} \right), \quad n \ge 0,
\end{equation}
where $\{ \hat p_n \}_{n = 0}^{\iy}$ are the Fourier coefficients:
$$
\hat p_n := \int_0^{\pi} \hat {\mathscr P}(t) \sin nt \, dt.
$$
Using Bessel's Inequality, \eqref{estPD}, and~\eqref{sm16}, we obtain the estimate~\eqref{sumrho}.
\end{proof}

\medskip

\textbf{Acknowledgements.} The author is grateful to Professors Namig Guliyev and Sergey Buterin for their valuable comments on this manuscript.

\noindent Natalia Pavlovna Bondarenko \\
1. Department of Applied Mathematics and Physics, Samara National Research University, \\
Moskovskoye Shosse 34, Samara 443086, Russia, \\
2. Department of Mechanics and Mathematics, Saratov State University, \\
Astrakhanskaya 83, Saratov 410012, Russia, \\
e-mail: {\it BondarenkoNP@info.sgu.ru}

\end{document}